\documentclass[a4paper]{ifacconf}

\usepackage{graphicx,amsmath,url}      
\usepackage[round]{natbib}             

\usepackage{amssymb}
\usepackage{color}

\newtheorem{dfn}{Definition}
\newtheorem{rmk}{Remark}

\DeclareMathOperator{\diag}{diag}
\DeclareMathOperator{\He}{He}
\newcommand{\bbr}{\mbox{$\mathbb{R}$}}
\newcommand{\bbs}{\mbox{$\mathbb{S}$}}
\newcommand{\bbd}{\mbox{$\mathbb{D}$}}
\newcommand{\mcl}{\mbox{$\mathcal{L}$}}
\newcommand{\mch}{\mbox{$\mathcal{H}$}}
\newcommand{\mcc}{\mbox{$\mathcal{C}$}}
\newcommand{\mcw}{\mbox{$\mathcal{W}$}}
\newcommand{\mcr}{\mbox{$\mathcal{R}$}}

\newcommand{\mcu}{\mbox{$\mathcal{U}$}}

\ifx\portugues\undefined
\def\portugues{0}
\fi

\if\portugues0
   \usepackage[english]{babel}
  \else
   \usepackage[spanish,brazil,english]{babel}
\fi

\usepackage[T1]{fontenc}

\usepackage[utf8]{inputenc}

\usepackage{ae}

\begin{document}


%

\begin{frontmatter}

\title{H$_1$-ISS Analysis and Boundary Control Design for Coupled Linear ODEs and Hyperbolic PDEs\thanksref{footnoteinfo}}

\thanks[footnoteinfo]{This work was partially supported by CAPES under grants 88887.629803/2021-00  and 88881.878833/2023-01 (SticAmSud), and CNPq under grant 303289/2022-8/PQ.}

\author[First]{F. V. Carvalho}
\author[Second]{D. Coutinho}
\author[Second]{G. de Andrade}
\author[Third]{C. Prieur}
\author[Third]{M. Maghenem}

\address[First]{Post-graduate Program in Automation and Systems Engineering, Universidade Federal de Santa Catarina, Florianópolis, SC, 88040-900, Brazil. (e-mail: eng.franthiescolly@gmail.com)}
\address[Second]{Department of Automation and Systems, Universidade Federal de Santa Catarina, Florianópolis, SC, 88040-900, Brazil. \\ (e-mails:  daniel.coutinho(gustavo.artur)@ufsc.br)}
\address[Third]{Université Grenoble Alpes, CNRS, Grenoble-INP, GIPSA-lab, F-38000, Grenoble, France. \\ (e-mails: christophe.prieur(mohamed.maghenem)@gipsa-lab.fr)}

\renewcommand{\abstractname}{{\bf Abstract:~}}

\begin{abstract}
This paper considers the general problem of the design of boundary controllers for distributed parameter systems. The control objectives are the asymptotic stability of the closed-loop system, as well as the disturbance attenuation of exogenous disturbances affecting the measurements and boundary conditions.  More specifically, this work studies input-to-state stability (ISS) and stabilization for systems governed by coupled linear ordinary differential equations (ODEs) and homogeneous linear hyperbolic partial differential equations (PDEs), where actuation, sensing, and disturbance inputs are located at the system boundaries. First, we extend a previously established Lyapunov-based ISS condition for a similar class of systems to the H$_1$ setting under mild assumptions on the disturbance inputs. This extension not only ensures that the H$_1$-norm of the system states driven by non-vanishing disturbances and compatible initial conditions are bounded, but also provides an upper bound on the system H$_1$-norm. Subsequently, these theoretical results are applied to derive stability analysis and  control design conditions expressed in terms of linear matrix inequalities. Numerical examples are provided to illustrate the effectiveness and potential of the proposed approach.
\end{abstract}

\begin{keyword}
Boundary stabilization, coupled ODE-PDE systems, input-to-state stability, LMIs.
\end{keyword}

\end{frontmatter}

\section{Introduction}

Many practical systems with complex dynamics are modeled using partial differential equations (PDEs). The stability analysis
and stabilization techniques for PDE-governed systems constitute a well-established area of research. A comprehensive discussions can be found in various textbooks as,
for instance, \cite{krstic08,bastin16,karafyllis18}.
More recently, significant attention has been directed towards developing stability and stabilization techniques that provide performance guarantees, particularly in the presence of (unknown but bounded) exogenous disturbances. In this context, most of the available results are predominantly addressed within the framework of input-to-state stability (ISS); see, e.g., \citep{mironchenko20}. From a practical perspective, actuators and sensors are typically positioned at the system boundaries. However, exogenous disturbances can affect the entire domain, the boundaries, or both, introducing significant challenge to the stability analysis, performance evaluation, and control synthesis. For instance, \cite{hasan14} proposes a state-feedback and state-observer for disturbance attenuation problems for disturbance attenuation in coupled linear hyperbolic PDEs, while \cite{tanwani16} addressed the input-to-state stabilization problem for linear hyperbolic systems 
using quantized boundary control. Additionally, \cite{wu22} investigated exponential stabilization of a wave equation subject to exogenous boundary disturbances, and \cite{Wang24} introduced an adaptive output feedback boundary control strategy for coupled hyperbolic PDEs with disturbances governed by an ordinary differential equation (ODE). Furthermore, \cite{suha22} provided a convex-optimization-based approach for designing controllers that minimize the effect of in-domain disturbances in hyperbolic systems. The latter works collectively provide diverse methodologies and shed light on challenges in addressing disturbance-rejection problems for infinite-dimensional systems \citep{deAndrade2024}. 

Recently, \cite{coutinho:2024} addressed the input-to-state stability and stabilization problems for coupled ODE-PDE systems
in the L$_2$-sense, where the disturbances are constrained to the class of continuously differentiable signals with bounded magnitude. In addition, only an upper bound on the average  value of the PDE states is provided. In this paper, we extend the previous results by considering stability and stabilization problems, for systems governed by coupled ODEs and hyperbolic PDEs, under the influence of magnitude-bounded disturbances 
affecting the system's boundaries. Our study is carried out within the H$_1$-framework, providing a more general approach to the stability analysis and boundary-control design under less restrictive assumptions on input perturbations. Furthermore, the proposed results are established in terms of linear matrix inequality (LMI) constraints providing a systematic procedure for stability analysis, performance verification and control design as illustrated by numerical examples.

The remainder of this paper is organized as follows. The problem of concern is described in Section~\ref{sec:poi}, while Section~\ref{sec:H1-ISS} introduces the notion of stability to be considered in this paper. In addition, Sections~\ref{sec:H1-sa} and~\ref{sec:H1-cd} develop input-to-state stability and stabilization conditions in the H$_1$-sense using Lyapunov-based arguments, respectively. The simulation results are shown in Section~\ref{sec:ie}, and Section~\ref{sec:cr} ends the paper with some concluding remarks.



\subsection*{Notation} \vspace{-2mm}
The notation used in this paper is standard. $\bbr$, $\bbr_+$, $\bbr^n$, and $\bbr^{n \times m}$ represent the set of real numbers, non-negative real numbers, $n$ dimensional real vectors, and $n \times m$ real matrices, respectively. $\bbs^n$ and $\bbd^n$ are the space of $n \times n$ symmetric and diagonal real matrices, respectively. An $n \times n$ identity matrix and an $n \times m$ matrix of zeros are denoted by $I_{n}$ and $0_{m \times n}$, respectively, with their dimensions being often omitted when they can be inferred from the context. For $W \in \bbr^{n \times n}$, $W^\top$ is its transpose, $\He\{W\} := W + W^\top$,  and the symbol $\star$ (for a block symmetric matrix)  represents a block matrix deduced by symmetry. $\mcc_{0}^n(Y)$ and $\mcc_{1}^n(Y)$ stand for the set of continuous and continuously differentiable $n$-dimensional vector functions, respectively, with $Y$ being a given normed space. $\mcl_2^{n}$ is the Hilbert space of square integrable ($n$-dimensional) vector functions over $[0,1]$, i.e., if $z \in \mcl_2^n$ then 
\begin{equation*}
\|z\|_{\mathcal{L}_2} = \left(\int_0^1 z(x)^T z(x) dx \right)^{1/2}  < \infty    .
\end{equation*}
For a vector signal $s : \bbr_+ \rightarrow \bbr^{n}$, its infinite norm is defined as
\begin{equation*}
\|s\|_\infty : = \left( \sup_{t\geq0}  s(t)^T s(t) \right)^{1/2} .    
\end{equation*}

\section{Problem Formulation} \label{sec:poi}

Consider the following class of linear PDE/ODE systems
\begin{align}
\partial_t \xi(t,z) & = - \Lambda \partial_z \xi(t,z) \label{eqn:pde-1} \\
\xi(t,0) & = H \xi(t,1) + B_u u(t) + B_v v(t) \label{eqn:pde-2} \\
\dot v(t) & = M v(t) + B_w w(t) \label{eqn:pde-3}  \\
y(t)  & = \xi(t,1)+ D_v v(t) , \label{eqn:pde-4} 
\end{align}
coupled with the following ODE
\begin{align}
&   u(t) = K_1 \eta(t) + K_2 y(t) \label{eqn:ode-1} \\
&   \dot \eta(t) = R \eta(t) + S y(t) , \label{eqn:ode-2}
\end{align}
where $t \in \bbr_+$, $z \in [0,1]$, $\xi(\cdot) = \begin{bmatrix} \xi_1(\cdot) & \cdots & \xi_n(\cdot) \end{bmatrix}^\top$, with $\xi_i \in \mcl_2, \ i=1,\ldots,n$, is the PDE state, $w \in \bbr^{n_w}$ is the exogenous disturbance input, $v \in \bbr^{n_v}$ is the disturbance filter state, $\eta \in \bbr^{n_\eta}$ is the state of the ODE defined in \eqref{eqn:ode-1} and \eqref{eqn:ode-2}, typically a dynamic boundary controller, and $y \in \bbr^n$ and $u \in \bbr^{n_u}$ are the signals coupling \eqref{eqn:pde-1}-\eqref{eqn:pde-4} with \eqref{eqn:ode-1}-\eqref{eqn:ode-2}. In addition, $\Lambda,H,B_u,B_v,M,B_w,D_v,K_1$, $K_2$, $R$ and $S$ are real matrices with appropriate dimensions, with $\Lambda$ being diagonal positive definite and $M$ being Hurwitz.  

The initial conditions of \eqref{eqn:pde-1}-\eqref{eqn:pde-4} and \eqref{eqn:ode-1}-\eqref{eqn:ode-2} are given by
\begin{equation} \label{eqn:ics}
 \xi(0,z) = \xi_0(z), \ z \in [0,1], \quad 
 v(0) = v_0 , \quad \eta(0) = \eta_0 
\end{equation}
with $\xi_0(z) \in \mcc_1^n([0,1])$, $v_0 \in \bbr^{n_v}$ and $\eta_0 \in \bbr^{n_\eta}$
satisfying the first-order compatibility conditions~\citep{bastin16}. 

Additionally, we consider the system defined in \eqref{eqn:pde-1}-\eqref{eqn:ics} equipped with the norm:
\begin{equation}
    \|(\xi, \partial_z\xi,v,\eta)\|_{X} : = \| \xi\|_{\mathcal{L}_2} 
    + \| \partial_z \xi\|_{\mathcal{L}_2} +
    \| v \|_2 + \| \eta \|_2 ,
\end{equation}
for every $(\xi,\partial_z \xi, v,\eta) \in X: = \mcl_2^n \times \mcl_2^n \times \bbr^{n_v} \times \bbr^{n_\eta}$. As an abuse of notation, we often refer to $\|(\xi,v,\eta)\|_{X}$ as the system's H$_1$-norm.

\begin{rmk}
The assumption that $\Lambda$ is positive definite is made without loss of generality. Indeed, we can always cast into this case through a change of variables; see 
\citep{caldeira18}.    
\end{rmk}


\begin{rmk}
The advantage of introducing the disturbance filter \eqref{eqn:pde-3} is that the boundary term inherits additional regularity in time. Indeed, since the exogenous signal $w$ does not directly enter the boundary equation in \eqref{eqn:pde-2} and it is first processed by the ODE in \eqref{eqn:pde-3}, one has that $v \in H^1(0,\infty)$ whenever $w \in L^2(0,\infty)$, and, in particular, $v$ is continuous on every compact time interval\footnote{Here, $L^2(0,\infty)$ and $H^1(0,\infty)$ are the usual spaces of time signals, i.e., finite energy and Sobolev spaces, respectively.}. Hence the boundary contribution $B_v v(t)$ has the temporal regularity typically required in the well-posedness analysis of \eqref{eqn:pde-1}--\eqref{eqn:pde-4} for $H^1$ solutions. Therefore, the filter regularizes the disturbance acting at the boundary and avoids imposing such regularity directly on $w$ itself.
\end{rmk}

Then, consider the following assumption on the disturbance input $w(t)$.
\begin{enumerate}
\item[{\bf A1}] The disturbance input signal $w(t)$ belongs to the following set:
\begin{equation}
\mcw : = \{ w \in (L^2(0,\infty))^{n_{w}} : w^\top w \leq 1 \} .
\end{equation}
\end{enumerate}

We stress that the size of $\mcw$ is fixed to a unitary maximum 
signals. In addition, the set $\mcw$ is normalized to one 
without loss of generality, since the columns of $B_w$ can be rescaled to handle different maximal input magnitudes.

\vskip 1mm

Now, we are ready to state the problems  to be addressed in this work.
\begin{itemize}

\item[{\bf P1}] When the matrices $K_1$, $K_2$, $R$, and $S$ are given, determine if the coupled PDE/ODE system as defined from \eqref{eqn:pde-1} to \eqref{eqn:ode-2}, with \eqref{eqn:ics}, is ISS, while providing an upper-bound on the system's H$_1$-norm; and

\item[{\bf P2}] For the coupled PDE/ODE system in \eqref{eqn:pde-1}-\eqref{eqn:pde-4}, determine a stabilizing boundary controller of the form \eqref{eqn:ode-1} and \eqref{eqn:ode-2} such that the resulting closed-loop system is ISS with respect to $w \in \mcw$, in the H$_1$-sense.    

\end{itemize}

\vskip 1mm

\section{H$_1$-ISS Characterization} \label{sec:H1-ISS}

To address problems {\bf P1} and {\bf P2}, we first define the ISS notion to be considered in this paper, as well as the related concept of ultimate-bound set. 

\vskip 1mm

\begin{dfn} \label{dfn:iss}
The coupled PDE/ODE system \eqref{eqn:pde-1}-\eqref{eqn:ode-2}
is said to be (exponentially) ISS with respect to $w \in \mcw$, if there exist positive scalars $\alpha, \rho$ and $\varrho$ such that its solution satisfies \vspace*{-1mm}
\begin{multline}
\|(\xi,v,\eta)\|_{X} \leq \rho \e^{- \alpha t} \|(\xi_0,v_0,\eta_0)\|_{X} + \varrho \|w\|_\infty  \\ \forall \ t \geq 0 . 
\end{multline}
\end{dfn}

\vskip 2mm

\begin{dfn}
Assume that the coupled PDE/ODE system defined in \eqref{eqn:pde-1} to \eqref{eqn:ode-2} is exponentially ISS with respect to $\mcw$, for some scalars $\alpha, \rho$ and $\varrho$. Then, the system's  ultimate-bound set (often referred as reachable set) is given by \vspace{-0.5mm}
\begin{multline} \label{eqn:dfn-mcr}
\mcu_{\mathcal{B}} : = \{ (\xi,\partial_z \xi,v,\eta) \in \mcl_2^n \times \mcl_2^n \times \bbr^{n_v} \times \bbr^{n_\eta} : \\
\|(\xi,v,\eta)\|_X^2 \leq \varrho \} .    
\end{multline}
\end{dfn}

\vskip 1mm

Next, to provide numerically tractable methods to address problems {\bf P1} and {\bf P2}, we introduce the following lemma.

\vskip 1mm

\begin{lem} \label{lem:iss-lyap}
Consider the coupled PDE/ODE system defined in \eqref{eqn:pde-1}-\eqref{eqn:ode-2} and suppose that assumption {\bf A1} holds. Let $V: \bbr_+ \times \mcl_2^n \times \mcl_2^n \times \bbr^{n_v} \times \bbr^{n_\eta} \rightarrow \bbr_+$ be continuously differentiable. If there exist positive scalars $\epsilon_1$, $\epsilon_2$, $\alpha$ and $\beta$ such that
\begin{align}
& \epsilon_1 \|(\xi,v,\eta)\|_X^2 \leq V(t,\xi,\partial_z \xi, v, \eta) \leq \epsilon_2 \|(\xi,v,\eta)\|_X^2 \label{eqn:lem1-1} \\
& \partial_t V(t,\xi,\partial_z \xi, v, \eta) + (2 \alpha + \beta) V(t,\xi,\partial_z \xi, v, \eta) \notag \\[1mm]
& \hspace{60mm} \leq \beta \|w\|_2^2 , \label{eqn:lem1-2}    
\end{align}
then the coupled system in  \eqref{eqn:pde-1}-\eqref{eqn:ode-2} is (exponentially) ISS with respect to $\mcw$. Moreover, 
for any compatible $(\xi_0,v_0,\eta_0)$, the resulting state trajectory, as $t \rightarrow \infty$, converges to the following region (which is an estimate of the ultimate bound set): 
\begin{multline} \label{eqn:mcr-h1iss}
\mcr : = \{ (\xi,\partial_z \xi,v,\eta) \in \mcl_2^n \times \mcl_2^n \times \bbr^{n_v} \times \bbr^{n_\eta} :\\ V(t,\xi,\partial_z \xi,v,\eta) \leq 1 \} .    
\end{multline}
\end{lem}

The proof of the above Lemma is straightforward based on knowing \citep[Lemma~1]{coutinho:2024} and it is omitted for brevity.

Lemma~\ref{lem:iss-lyap} can be applied to deduce a bound on the sup norm of $\xi(t,z)$ given by 
\begin{equation}
\| \xi\|_{\sup} : = \sup_{z \in [0,1]}
\sqrt{\xi(t,z)^\top \xi(t,z)} .    
\end{equation}
To illustrate this point, let 
\begin{multline}
\mcr_0 : = \{ (\xi,\partial_z \xi,v,\eta) \in \mcl_2^n \times \mcl_2^n \times \bbr^{n_v} \times \bbr^{n_\eta} : \\  V(t,\xi,\partial_z \xi,v,\eta) \leq 1 + c \},
\end{multline}
for some $c >0$, and assume $(\xi_0,v_0,\eta_0) \in \mcr_0$. Then, \eqref{eqn:lem1-1} and \eqref{eqn:lem1-2} imply that any state trajectory (driven by compatible initial conditions and inputs) starting from $\mcr_0$ remains in $\mcr_0$ for all $t \geq 0$ and converges to $\mcr$ as $t \rightarrow \infty$. 
From the fact that \citep{zheng2018input}
\begin{equation}
\xi(t,z)^\top \xi(t,z) \leq 2 \| \xi \|_{\mathcal{L}_2}^2 
+ \|\partial_z \xi \|_{\mathcal{L}_2}^2 , \ \forall \ z \in [0,1], 
\end{equation}
it follows that\vspace{-1mm}
\begin{multline}
\| \xi\|_{\sup}^2 \leq 2 \| \xi \|_{\mathcal{L}_2}^2 
+ \|\partial_z \xi \|_{\mathcal{L}_2}^2 \leq 2 \big( \|\xi\|_{\mathcal{L}_2} + \\ \|\partial_z \xi\|_{\mathcal{L}_2} 
+ \| v\|_2 + \|\eta \|_2 \big)^2 = 2 \|(\xi,v,\eta)\|_X^2.
\end{multline}
Next, assuming that $\|(\xi,v,\eta)\|_{X} \in \mcr_0$, for all $t \geq 0$, we obtain
\begin{equation}
    \|\xi\|_{\sup} \leq \sqrt{2(1+c)/\epsilon_1}  \quad \forall \ t \geq 0. 
\end{equation}
In addition, from the fact that $\|(\xi,\partial_z\xi,v,\eta)\|_X \rightarrow \mcr$ as $t \rightarrow \infty$, it follows the ultimate bound on the sup norm, denoted as $\|\xi\|_{\sup}^{(ub)}$, satisfies     
\begin{equation}    
\|\xi\|_{\sup}^{(ub)} \leq
\sqrt{2/\epsilon_1}  ,
\end{equation}
since $\xi(t,z)$ converges to $\mcr$ as $t \rightarrow \infty$.

\section{Stability Analysis} \label{sec:H1-sa}

In this section, we propose a numerically-tractable (LMI) condition to guarantee ISS with respect to $\mcw$ of the coupled system of \eqref{eqn:pde-1}-\eqref{eqn:ode-2}, as well as an optimization problem  to derive a tight estimate of the system's ultimate-bound set $\mcu_{\mathcal B}$. To this end, we consider  the following Lypuanov function candidate:
\begin{equation} \label{eqn:Vt-1+4}
 V = V(\xi,\partial_z \xi,v,\eta) = 
 \sum_{i=1}^4 V_i ,   
\end{equation}
where
\begin{align}
V_1 & = \int_0^1 \e^{-\mu_1 z} \xi(t,z)^\top P_1 \xi(t,z) dz, \notag \\[2mm]
V_2 & = \int_0^1 \e^{-\mu_2 z} \partial_z \xi(t,z)^\top P_2 \partial_z \xi(t,z) dz, \notag \\[2mm]
V_3 & = \eta(t)^\top P_3 \eta(t), \notag \\[2mm]
V_4 & = v(t)^\top P_4 v(t), \notag
\end{align}
with $\mu_1$ and $\mu_2$ given positive real scalars and $P_1 \in \bbd^{n}$, $P_2 \in \bbd^{n}$, $P_3 \in \bbs^{n_v}$ and $P_4 \in \bbs^{n_\eta}$ positive definite matrices to be determined. 

\vskip 2mm

\begin{thm} \label{thm:h1-iss}
Consider the coupled PDE/ODE system defined in \eqref{eqn:pde-1}-\eqref{eqn:ode-2}, and suppose that Assumption {\bf A1} holds. Let $\alpha$ and $\beta$ be two given positive scalars such that 
\begin{equation} \label{eqn:thm1-LMI-1}
    2 \alpha + \beta \leq \e^{-1} \lambda ,
\end{equation}
where $\lambda$ is the smallest eigenvalue of $\Lambda$. Suppose there exist positive definite matrices $P_1 \in \bbd^{n}$, $P_2 \in \bbd^{n}$, $P_3 \in \bbs^{n_v}$ and $P_4 \in \bbs^{n_\eta}$ such that the following LMI holds
\begin{equation} \label{eqn:thm1-LMI-2}
\begin{bmatrix} 
\Psi_a & \star & \star \\
P_1 \Lambda \Phi_a & - P_1 \Lambda & \star \\
P_2 \Phi_b & 0 & - P_2 \Lambda 
\end{bmatrix} < 0 ,
\end{equation}
where  
\begin{align}
\Psi_a & : = \left[ \psi_{ij} \right]_{i,j=1,\ldots,5}, \ \psi_{ij} = \psi_{ji} , \label{eqn:thm1-not-2} \\[2mm]
\Phi_a & := \begin{bmatrix} 0 & (H+B_u K_2 ) & B_u K_1 & (B_v + B_u K_2 D_v) & 0 \end{bmatrix} , \notag \\[2mm]
\Phi_b & : = \left[ \begin{matrix} (H+B_u K_2) \Lambda & - B_u K_1 S &  - B_u K_1 R \end{matrix} \right. \notag \\
& \hspace{5mm}    \left. \begin{matrix} - (B_vM \!+\! B_u K_1 S D_v \!+\! B_u K_2 D_v M)  & -B_v B_w \end{matrix} \right] , \notag
\end{align}
with 
\begin{align}
\vartheta & := \alpha+0.5 \beta, \ 
\psi_{11} : = - \e^{-1} \Lambda P_2, \  
\psi_{22} : = - \e^{-1} \Lambda P_1 , \notag \\
\psi_{32} & : = P_3 S, \ \psi_{33} := \He\{P_3 R+ \vartheta P_3\},  \ \psi_{34} := P_3 S D_v,  \notag \\ 
\psi_{44} & := \He\{P_4 M+ \vartheta P_4\} , \  
\psi_{45} := P_4 B_w , \ 
\psi_{55} := - \beta I_{n_w} , \notag   
\end{align}
and the remaining $\psi_{ij}$ are matrices of zeros having appropriate dimensions. 

Then, the coupled system defined in  \eqref{eqn:pde-1}-\eqref{eqn:ode-2} is (exponentially) ISS with respect to $\mcw$.  Moreover, the set $\mcr$ defined in \eqref{eqn:mcr-h1iss}
is an estimate of the system's ultimate-bound set, where $V$ is as defined in \eqref{eqn:Vt-1+4} with $\mu_1=\mu_2=1$.  
\end{thm}

\vskip 1mm

{\em Proof:} Let $V(\xi,\partial_z \xi,v,\eta)$ be as defined in \eqref{eqn:Vt-1+4}. Firstly, since $P_1,P_2,P_3$ and $P_4$ are definite positive, it is straightforward to show that \eqref{eqn:lem1-1} holds with $\epsilon_1$ and $\epsilon_2$ given by
\begin{align} 
\epsilon_1 & = \min\{\e^{-\mu_1}p_1,\e^{-\mu_2}p_2,p_3,p_4\} , \notag \\
\epsilon_2 & = \max\{q_1,q_2,q_3,q_4\}, \label{eqn:proof-thm1-0}
\end{align}
where $p_i$ and $q_i$ are respectively the smallest and largest eigenvalues of $P_i$, for $i=1,\ldots,4$. 

Next, we demonstrate that the conditions in \eqref{eqn:thm1-LMI-1} and \eqref{eqn:thm1-LMI-2} imply that \eqref{eqn:lem1-2} holds. To this end, we compute the time derivative of $V$ along the system's trajectories in the following steps:

\vskip 1mm

\noindent $(i)$ Time-derivative of $V_1$ $\Rightarrow$
\begin{align}
    \dot V_1 & = 2 \int_0^1 \e^{-\mu_1 z} \xi(t,z)^\top P_1 \partial_t \xi(t,z) dz \notag \\
    & = - 2 \int_0^1 \e^{-\mu_1 z} \xi(t,z)^\top P_1 \Lambda \partial_z \xi(t,z) dz 
    \notag \\ 
    & = - \left[ \e^{-\mu_1 z} \xi(t,z)^\top \Lambda P_1 \xi(t,z) \right]_{z=0}^{z=1} \notag \\
    & \hspace{15mm} - \mu_1 \int_0^1 \e^{-\mu_1 z} \xi(t,z)^\top P_1 \Lambda \xi(t,z) dz \notag \\
    & = - \e^{-\mu_1} \xi(t,1)^\top \Lambda P_1 \xi(t,1) + \xi(t,0)^\top \Lambda P_1 \xi(t,0) \notag \\ & \hspace{12mm} - \mu_1 \int_0^1 \e^{-\mu_1 z} \xi(t,z)^\top P_1 \Lambda \xi(t,z) dz . \label{eqn:proof-thm1-1}
\end{align}

Taking the boundary conditions into account, we obtain from \eqref{eqn:proof-thm1-1} the following:  
\begin{multline} \label{eqn:proof-thm1-2}
\dot V_1 = \phi_1^\top \big( \Psi_1 + \Phi_1^\top P_1 \Lambda \Phi_1 \big) \phi_1 \\ -  \int_0^1 
\mu_1 \e^{-\mu_1 z} \xi(t,z)^\top P_1 \Lambda \xi(t,z) dz ,
\end{multline}
where
\begin{align}
    \phi_1^\top & := \begin{bmatrix} \xi(t,1)^\top & \eta(t)^\top & v(t)^\top \end{bmatrix}, \notag \\
    \Phi_1 & := \begin{bmatrix} (H+B_u K_2) & B_u K_1 & (B_v + B_u K_2 D_v ) \end{bmatrix}, \notag \\
    \Psi_1 & := \diag \{ -\e^{-\mu_1} \Lambda P_1, 0_{n_\eta} , 0_{n_v} \} . \notag 
\end{align}

\vskip 1mm

\noindent $(ii)$ Time-derivative of $V_2$ $\Rightarrow$
\begin{align}
    \dot V_2 & = 2 \int_0^1 \e^{-\mu_2 z} \partial_z \xi(t,z)^\top P_2 \partial_t \partial_z \xi(t,z) dz \notag \\
    & = - 2 \int_0^1 \e^{-\mu_2 z} \partial_z \xi(t,z)^\top P_2 \Lambda \partial_z \big(\partial_z \xi(t,z) \big) dz \notag \\ 
    & = - \left[ \e^{-\mu_2 z} \partial_z \xi(t,z)^\top \Lambda P_2 \partial_z \xi(t,z) \right]_{z=0}^{z=1} \notag \\
    & \hspace{5mm} - \mu_2 \int_0^1 \e^{-\mu_2 z} \partial_z \xi(t,z)^\top P_2 \Lambda \partial_z \xi(t,z) dz \notag \\
    & = - \e^{-\mu_2} \partial_z \xi(t,1)^\top \Lambda P_2 \partial_z \xi(t,1) \notag \\
    & \hspace{5mm} + \partial_z \xi(t,0)^\top \Lambda P_2 \partial_z \xi(t,0) \notag \\ 
    & \hspace{5mm} - \int_0^1 \mu_2 \e^{-\mu_2 z} \partial_z \xi(t,z)^\top P_2 \Lambda \partial_z \xi(t,z) dz . \label{eqn:proof-thm1-3}
\end{align}

Now, we have to obtain a boundary condition coupling $\partial_z \xi(t,1)$ and $\partial_z \xi(t,0)$. Notice from \eqref{eqn:pde-2} that the time derivative of $\xi(t,0)$ can be cast as follows
\begin{align}
\partial_t \xi(t,0) & = H \partial_t \xi(t,1) + B_u \dot u + B_v \dot v \notag  \\
& = H \partial_t \xi(t,1) + B_u \big( K_1 \dot \eta + K_2 \dot y(t) \big) \notag \\
& \hspace{25mm} + B_v \big( M v(t) + B_w w(t) \big)
\notag \\
& = (H + B_u K_2) \partial_t \xi(t,1) + B_u K_1 S \xi(t,1) \notag \\
& \hspace{3mm} + \big( B_v M +B_u K_1 S D_v + B_u K_2 D_v M \big) v(t) \notag 
\end{align}
\begin{align}
& \hspace{5mm} + B_u K_1 R \eta(t) + B_v B_w w(t) 
\label{eqn:proof-thm1-4} .  
\end{align}
Hence, taking \eqref{eqn:pde-1} into account, we obtain the following from the above expression
\begin{equation} \label{eqn:proof-thm1-5} 
    \partial_z \xi(t,0) = \Lambda^{-1} \Phi_2 \phi_2 ,
\end{equation}
where 
\begin{align}
\Phi_2 & := \left[ \begin{matrix} (H+B_uK_2)\Lambda & -B_uK_1S & - B_u K_1 R \end{matrix} \right. \notag \\
& \hspace{15mm} \left. \begin{matrix} - (B_v M+ B_u K_1 S D_v) & - B_v B_w \end{matrix} \right] , \label{eqn:phi2} \\
\phi_2^\top & := \begin{bmatrix} \partial_z \xi(t,1)^\top & 
\xi(t,1)^\top & \eta(t)^\top & v(t)^\top & w(t)^\top \end{bmatrix} . \notag 
\end{align}

Therefore, applying \eqref{eqn:proof-thm1-4} and \eqref{eqn:proof-thm1-5} into  \eqref{eqn:proof-thm1-3} leads to 
\begin{multline} \label{eqn:proof-thm1-6}
\dot V_2 = \phi_2^\top \big( \Psi_2 + \Phi_2^\top P_2 \Lambda^{-1} \Phi_2 \big) \phi_2 \\ - \int_0^1 \mu_2 \e^{-\mu_2 z} \partial_z \xi(t,z)^\top P_2 \Lambda \partial_z \xi(t,z) dz , 
\end{multline}
where
\begin{equation}
    \Psi_2 = \diag\{ - \e^{-\mu_2} \Lambda P_2, 0_{n},0_{n_\eta},0_{n_v},0_{n_w} \} .
\end{equation}

\vskip 1mm

\noindent $(iii)$ Time-derivative of $V_3$ $\Rightarrow$
\begin{align} 
\dot V_3 & = 2 \eta(t)^\top P_3 \dot \eta(t) \notag \\
& = 2 \eta(t)^\top P_3 \big( R \eta(t) + S \xi(t,1) + S D_v v(t)\big) \notag \\
& = \phi_2^\top \Psi_3 \phi_2 ,
\label{eqn:proof-thm1-7}
\end{align}
where 
\begin{equation} \notag 
   \Psi_3 := \begin{bmatrix}
       0 & \star & \star & \star & \star \\
       0 & 0     & \star & \star & \star \\
       0 & P_3 S & \He\{P_3 R\} & P_3 S D_v & \star \\
       0 & 0     & \star              & 0 & \star \\
       0 & 0     & 0    & 0     & 0  
   \end{bmatrix} .
\end{equation}

\vskip 1mm

\noindent $(iv)$ Time-derivative of $V_4$ $\Rightarrow$
\begin{align} 
\dot V_4 & = 2 v(t)^\top P_4 \dot v(t) = 
2 v(t)^\top P_4 \big( M v(t) + B_w w(t)\big) \notag \\
& = \phi_2^\top \Psi_4 \phi_2 , \label{eqn:proof-thm1-8} 
\end{align}
where 
\begin{equation} \notag 
   \Psi_4 := \begin{bmatrix}
       0 & \star & \star & \star & \star \\
       0 & 0     & \star & \star & \star \\
       0 & 0     & 0     & \star & \star \\
       0 & 0     & 0     & \He\{P_4M\} & P_4 B_w \\
       0 & 0     & 0     & \star     & 0  
   \end{bmatrix} .
\end{equation}

\vskip 2mm

Therefore, in the light of \eqref{eqn:proof-thm1-2}, \eqref{eqn:proof-thm1-6}, \eqref{eqn:proof-thm1-7} and \eqref{eqn:proof-thm1-8}, the ISS condition of \eqref{eqn:lem1-2} can be cast as follows:
\begin{multline} \label{eqn:proof-thm1-9}
 \partial_t V + (2\alpha+\beta) V - \beta w(t)^\top w(t) \\ =  
 \phi_2^\top \big( \Psi_b(\mu_1,\mu_2) + \Phi_a ^\top P_1 \Lambda \Phi_a + \Phi_2^\top P_2 \Lambda^{-1} \Phi_2 \big) \phi_2 \\
 + V_a + V_b < 0 \quad \forall w \in \mcw ,      
\end{multline}
where 
\begin{align}
\Psi_b(\mu_1,\mu_2) & := \left[ \hat{\psi}_{ij} \right]_{i,j=1,\ldots,5}, \ \hat{\psi}_{ij} = \hat{\psi}_{ji} , \notag \\
\Phi_a & := \begin{bmatrix} 0 & \Phi_1 & 0 \end{bmatrix} , \notag \\
V_a & := \int_0^1 \xi(t,z)^\top \big( (2 \alpha+\beta) P_1 \notag \\
& \hspace{20mm} - \mu_1 \e^{-\mu_1 z} P_1 \Lambda \big) \xi(t,z) dz , \notag \\
V_b & := \int_0^1 \partial_z \xi(t,z)^\top \big( (2 \alpha+\beta) P_2  \notag \\
& \hspace{15mm} - \mu_2 \e^{-\mu_2 z} P_2 \Lambda \big) \partial_z \xi(t,z) dz , \notag
\end{align}
with $\hat{\psi}_{ij}$ being equal to the matrices $\psi_{ij}$ defined in \eqref{eqn:thm1-not-2} excepting 
\begin{align*}
\hat{\psi}_{11} & := - \e^{-\mu_2} \Lambda P_2, \quad \hat{\psi}_{22} := - \e^{-\mu_1} \Lambda P_1 .     
\end{align*}

Hence, the following matrix inequalities are sufficient conditions for \eqref{eqn:proof-thm1-9} to hold:
\begin{align}
    \Psi_b(\mu_1,\mu_2) + \Phi_a^\top P_1 \Lambda \Phi_a + \Phi_2^\top P_2 \Lambda^{-1} \Phi_2 < 0 , \label{eqn:proof-thm1-12} \\
    (2 \alpha+\beta) P_1 - \mu_1 \e^{-\mu_1} P_1 \Lambda \leq 0 , \label{eqn:proof-thm1-13} \\
    (2 \alpha+\beta) P_2 - \mu_2 \e^{-\mu_2} P_2 \Lambda \leq 0 . \label{eqn:proof-thm1-14}
\end{align}

Now, notice that \eqref{eqn:proof-thm1-13} and \eqref{eqn:proof-thm1-14} are equivalent to
\begin{equation} \label{eqn:proof-thm1-15}
    \mu_i \e^{-\mu_i} \Lambda - (2 \alpha+\beta) \geq 0, \ i=1,2 ,
\end{equation}
since $P_1$ and $P_2$ are diagonal positive definite matrices. Therefore, one can set $\mu_1 = \mu_2 = \mu$ without loss of generality. In addition, by noting that the maximum of $\mu \e^{-\mu}$ is achieved for $\mu=1$, the following sufficient condition is obtained for \eqref{eqn:proof-thm1-15}
\begin{equation} \label{eqn:proof-thm1-16}
    \e^{-1} \Lambda - (2 \alpha+\beta) \geq 0, 
\end{equation}
which is implied by \eqref{eqn:thm1-LMI-1}. 

Next, from Schur's complement and taking into account that $\Psi_a = \Psi_b(1,1)$, the LMI in \eqref{eqn:thm1-LMI-2} ensures that \eqref{eqn:proof-thm1-12} is satisfied, which completes the proof.   \hfill $\Box$

\vskip 3mm

\section{Boundary Control Design} \label{sec:H1-cd}

In this section, we design a (dynamic) boundary controller 
of the form \eqref{eqn:ode-1}-\eqref{eqn:ode-2} to 
guarantee ISS (in the $\mch_1$ sense) for the coupled PDE/ODE system \eqref{eqn:pde-1}-\eqref{eqn:pde-4} with respect to $\mcw$. Furthermore, we are interested in obtaining a tight ultimate-bound set. We stress that the direct application of Theorem~\ref{thm:h1-iss} for control design leads to conditions that are difficult to solve numerically, since the inequalities are nonlinear in the control variables. In order to obtain 
an LMI-based condition for control design, we impose the following constraints on the linear filter in \eqref{eqn:pde-3}-\eqref{eqn:pde-4}. 

\vskip 1mm

\begin{itemize}

\item[{\bf A2}] The matrices $M$, $B_w$ and $D_v$ are diagonal and $n_\eta = n_v = n_w =n$.

\end{itemize}

\vskip 1mm

The above assumption implies that: $(i)$ the number of disturbance signals is equal to the number of measurements; and $(ii)$ the disturbance filter consists on single-input-single-output 1-st order filters. The first implication is not a problem, since we can always inflate $w$ with fictitious disturbances, and the second one mean that each input disturbance channel is filtered by decoupled first-order filters of the form
\begin{equation}
    \dot v_i(t) = - m_i v_i(t) + b_i w_i(t) , \ i=1,\ldots,n ,
\end{equation}
where $v_i$ and $w_i$ are the $i$-th entries of $v$ and $w$, respectively, and $m_i$ and $b_i$ are the $i$-th diagonal elements of $M$ and $B_w$, respectively.

Another difficulty appears when determining a relationship between $\partial_z\xi(t,0)$ and $\partial_z \xi(t,1)$, which comes out when computing the term $\dot V_2$ in \eqref{eqn:proof-thm1-3}. To overcome this problem, we express $\partial_z\xi(t,0)$ as: 
\begin{equation} \label{eqn:partialzxi0}
\partial_z \xi(t,0) = \Lambda^{-1} \tilde{\Phi}_2 \phi ,
\end{equation} 
where
\begin{align}
\tilde{\Phi}_2 & : = \left[ \begin{matrix} 
H \Lambda & 0 & 0 &\! -B_v M &\! - B_v B_w &\! - B_u K_1 & - B_u K_2 \end{matrix} \right] , \notag \\
\phi & := \begin{bmatrix}
    \phi_2^\top & \dot \eta(t)^\top & \dot y(t)^\top 
\end{bmatrix}^\top , \label{eqn:phi}
\end{align}
with $\phi_2$ as defined in \eqref{eqn:phi2}. 
Notice that elements of $\phi$ are related by the equations \eqref{eqn:pde-3}-\eqref{eqn:ode-2}, \eqref{eqn:proof-thm1-4} and \eqref{eqn:proof-thm1-5}, which can be expressed by the following equality constraints:
\begin{equation} \label{eqn:Gammas}
\Gamma_a \phi = 0 , \qquad \Gamma_b \phi = 0,
\end{equation}
where
\begin{align}
\Gamma_a & := \begin{bmatrix} 
0 & S & R & S D_v & 0 & -I & 0     
\end{bmatrix} , \notag \\
\Gamma_b & := \begin{bmatrix}
-\Lambda & 0 & 0 & D_v M & D_v B_w & 0 & -I  
\end{bmatrix}^\top , \label{eqn:gammas}
\end{align}
from the fact that 
\begin{multline} \notag 
\dot y = \partial_t\xi(t,1) + D_v \dot v \\ 
= - \Lambda \partial_z \xi(t,1) + D_v M v + D_v B_w w .    
\end{multline}

Then, by considering similar steps as in the proof of Theorem~\ref{thm:h1-iss} with additional algebraic manipulations, we obtain the following result.

\vskip 2mm

\begin{thm} \label{thm:h1-iss-bcd}
Consider the coupled PDE/ODE system defined in \eqref{eqn:pde-1}-\eqref{eqn:ode-2},  and suppose that Assumptions {\bf A1} and {\bf A2} hold. Let $\alpha$ and $\beta$ be two given positive scalars such that 
\begin{equation} \label{eqn:thm1-LMI-1-bcd}
    2 \alpha + \beta \leq \e^{-1} \lambda ,
\end{equation}
where $\lambda$ is the smallest eigenvalue of $\Lambda$. Suppose there exist matrices $Q_1 \in \bbd^{n}$, $Q_2 \in \bbd^{n}$, $G \in \bbd^{n}$, $P_3 \in \bbs^{n_v}$, $P_p \in \bbs^{n_\eta}$, with $P_p >0$, $K_1 \in \bbr^{n \times n}$, $K_p \in \bbr^{n \times n}$, $R_p \in \bbr^{n \times n}$, $S_p \in \bbr^{n \times n}$, and $L_p \in \bbr^{n \times n}$ such that
\begin{equation} \label{eqn:thm1-LMI-2-bcd}
\begin{bmatrix} 
\tilde{\Psi}_a & \star & \star \\
\tilde{\Psi}_b & - Q_1 \Lambda & \star \\
\tilde{\Psi}_c & 0 & - Q_2 \Lambda 
\end{bmatrix} < 0
\end{equation}
where
\begin{align}
\tilde{\Psi}_a & := \begin{bmatrix} \tilde{\psi}_{ij}
\end{bmatrix}_{i,j=1,\ldots,7} , \quad \tilde{\psi}_{ij} = \tilde{\psi}_{ji} , \label{eqn:thm3-not-2} \\[2mm]
\tilde{\Psi}_b & := \Lambda \left[ \begin{matrix} 0 & (HG+B_u K_p) & B_u K_1 & B_u K_p D_v \end{matrix} \right. \notag \\
& \hspace{55mm} \left. \begin{matrix} B_w G & 0 & 0 \end{matrix} \right] ,  \notag \\
\tilde{\Psi}_c & := \left[ \begin{matrix} H \Lambda G & 0 & 0 & - B_v M G  & -B_v B_w G \end{matrix} \right. \notag \\
& \hspace{45mm} \left. \begin{matrix} - B_u K_1 & - B_u K_p \end{matrix} \right] , \notag \\[2mm]
\tilde{\psi}_{11} & := \e^{-1} \Lambda (Q_2 - 2G), \quad \tilde{\psi}_{22} := \e^{-1} \Lambda (Q_1 - 2G), \notag \\[2mm]
\tilde{\psi}_{32} & := S_p , \ 
\tilde{\psi}_{33} := \He\{R_p + \vartheta P_3\} , \
\tilde{\psi}_{43} := D_v S_p^\top , \notag \\[2mm]
\tilde{\psi}_{44} & := \! \He\{ P_p M + \vartheta P_p \} , \quad \tilde{\psi}_{54} = B_w P_p , \notag \\[2mm] 
\tilde{\psi}_{55} & := \beta(I \!-\! 2G), \quad 
\tilde{\psi}_{62} := S_p, \quad  
\tilde{\psi}_{63} := R_p , \notag \\[2mm] \tilde{\psi}_{64} & := S_p D_v, \quad 
\tilde{\psi}_{66} := -2 P_3 , \quad
\tilde{\psi}_{71} := - L_p \Lambda , \notag 
\end{align}
\begin{align}
\tilde{\psi}_{74} & := L_p D_v M , \ 
\tilde{\psi}_{75} := L_p D_v B_w , \
\tilde{\psi}_{77} & := -\!L_p \!-\! L_p^\top , \notag
\end{align}
with the remaining $\tilde{\psi}_{ij}$ being matrix of zeros and $\vartheta := \alpha+0.5 \beta$. Then, the coupled system defined in  \eqref{eqn:pde-1}-\eqref{eqn:ode-2}, with  
\begin{equation} \notag 
K_1, \ K_2 := K_p G^{-1}, \ R := P_3^{-1} R_p , \ S := P_3^{-1} S_p G^{-1} , 
\end{equation}
is (exponentially) ISS with respect to $\mcw$.
Moreover, $\mcr$ given by\eqref{eqn:mcr-h1iss}, where
$V$ is as defined in \eqref{eqn:Vt-1+4} with $P_1 = Q_1^{-1}$, $P_2 = Q_2^{-1}$, $P_3$, $P_4 = G^{-1} P_p G^{-1}$ and $\mu_1=\mu_2=1$, is an estimate of the { system's ultimate-bound set.}  
\end{thm}

\vskip 1mm

\noindent {\em Proof:}
Assume that there exists a solution to \eqref{eqn:thm1-LMI-2-bcd}. From the fact that $- \bar{G}^\top Q^{-1} \bar{G} \leq Q - \bar{G} - \bar{G}^\top$ for any $Q>0$ and any nonsingular $\bar{G}$, the LMI in \eqref{eqn:thm1-LMI-2-bcd} implies that:
\begin{equation} \label{eqn:thm2-1}
\begin{bmatrix} 
\bar{\Psi}_a & \star & \star \\
\tilde{\Psi}_b & - Q_1 \Lambda & \star \\
\tilde{\Psi}_c & 0 & - Q_2 \Lambda 
\end{bmatrix} < 0 ,
\end{equation}
where
\begin{align}
\bar{\Psi}_a & := \begin{bmatrix} \bar{\psi}_{ij} \end{bmatrix}_{i,j=1,\ldots,7} , \quad 
\bar{\psi}_{ij} = \bar{\psi}_{ji},  \label{eqn:thm2-2} \\[2mm]
\bar{\psi}_{11} & \!:=\! - \frac{\Lambda G Q_2^{-1} G}{\e}, \  
\bar{\psi}_{22} \!:=\! - \frac{\Lambda G Q_1^{-1}G}{\e} , \  \bar{\psi}_{55} \!:=\! - \beta G G, \notag
\end{align}
with  $\bar{\psi}_{ij} = \tilde{\psi}_{ij}$ for the remaining block matrices and $\tilde{\psi}_{ij}$ as given in \eqref{eqn:thm3-not-2}. 

Next, by applying the Schur's complement to \eqref{eqn:thm2-1}, we obtain
\begin{equation} \label{eqn:thm2-3}
\bar{\Psi}_a + \bar{\Psi}_b^\top \Lambda Q_1^{-1} \bar{\Psi}_b + \bar{\Psi}_c^\top \Lambda Q_2^{-1} \bar{\Psi}_c < 0 ,
\end{equation}
where $\bar{\Psi}_b := \Lambda^{-1} \tilde{\Psi}_b$ and
$\bar{\Psi}_c := \Lambda^{-1} \tilde{\Psi}_c$. 

Now, let $P_1 := Q_1^{-1}$ and $P_2 := Q_2^{-1}$ and consider the following parametrizations:
\begin{equation} \label{eqn:thm2-4}
\begin{array}{c} 
L_p := G L_b G, \ S_p := P_3 S G, \ K_p := K_2 G, \\[1mm]
R_p := P_3 R, \quad P_p := G P_4 G. 
\end{array}
\end{equation}
From the above, one can cast \eqref{eqn:thm2-3} 
into
\begin{equation} \label{eqn:thm2-5}
    G_a \big( \breve{\Psi}_a + \bar{\Psi}_b^\top \Lambda P_1 \bar{\Psi}_b + \bar{\Psi}_c^\top \Lambda P_2 \bar{\Psi}_c \big) G_a < 0  ,
\end{equation}
where
\begin{align}
G_a & := \diag\{G,G,I,G,G,I,G\},   \notag \\
\breve{\Psi}_a & := \begin{bmatrix} \breve{\psi}_{ij} \end{bmatrix}_{i,j=1,\ldots,7}, \quad \breve{\psi}_{ij} = \breve{\psi}_{ji}, \label{eqn:thm2-6} \\
\breve{\psi}_{11} & := - \e^{-1} \Lambda P_2 , \quad 
\breve{\psi}_{22} := - \e^{-1} \Lambda P_1 , \quad 
\breve{\psi}_{32} := P_3 S , \notag \\ 
\breve{\psi}_{33} & := \He\{ P_3 R + \vartheta P_3 \} , \quad \breve{\psi}_{43} := D_v^\top S^\top P_3 , \notag \\
\breve{\psi}_{44} & := \He\{ P_4 M + \vartheta P_4\} , \quad \breve{\psi}_{54} := B_w^\top P_4 , \notag \\ 
\breve{\psi}_{55} & := - \beta I , \quad
\breve{\psi}_{62} := P_3  S , \quad 
\breve{\psi}_{63} := P_3 R , \notag \\ 
\breve{\psi}_{64} & := P_3 S D_v , \quad 
\breve{\psi}_{66} := - 2 P_3 , \quad 
\breve{\psi}_{71} := - L_b \Lambda , \notag \\ 
\breve{\psi}_{74} & := L_b D_v M , \  
\breve{\psi}_{75} := L_b D_v B_w , \ 
\breve{\psi}_{77}  := - \He\{L_b\} , \notag
\end{align}
with the remaining blocks of $\breve{\Psi}_a$ being matrices of zeros. 

Notice from \eqref{eqn:thm2-1} that $G$ is non-singular and therefore $\breve{\Psi}_a + \bar{\Psi}_b^\top \Lambda P_1 \bar{\Psi}_b + \bar{\Psi}_c^\top \Lambda P_2 \bar{\Psi}_c < 0$, which pre- and post-multiplied by $\phi^\top$ and $\phi$, respectively, leads to \eqref{eqn:proof-thm1-9}, with $\mu_1 = \mu_2 = 1$, in view of \eqref{eqn:Gammas}. The rest of this proof 
is straightforward under the proof of Theorem~\ref{thm:h1-iss}. \hfill $\Box$

\vskip 2mm

\section{Illustrative Examples} \label{sec:ie}

In the following, we provide two academic examples to illustrate the potentials of the proposed approach for stability analysis and boundary control design for hyperbolic PDEs in the H$_1$-sense.

\subsection{Stability Analysis}

Let the coupled PDE/ODE system be as defined in \eqref{eqn:pde-1}-\eqref{eqn:ode-2} with  
\begin{align}
\Lambda & = \begin{bmatrix}
               7  &  0  \\
               0  & 6
    \end{bmatrix}, \quad  H = \begin{bmatrix}
               0  &  -0.21  \\
               -1.15  & 0
    \end{bmatrix},  \notag \\
B_u & = \begin{bmatrix} 1 & 0 \\ 0 & 1 \end{bmatrix} , \quad B_{v} = \frac{1}{2} \begin{bmatrix} 
    1 \\ 1 \end{bmatrix} ,  \quad 
D_v = 10^{-2} \begin{bmatrix} 0 \\ 1 \end{bmatrix} , \notag \\
K_1 &= 10^{-3} \begin{bmatrix}
-0.1170 & 0.2198 \\ 0.2267 & -0.4266  
\end{bmatrix} , \label{eq:parametrosl7_exemplo1}
\end{align}
\begin{align}
K_2 & = \begin{bmatrix}
 0.0000 & 0.2100 \\ 1.1500 & 0.0000
\end{bmatrix} , \quad
R = \begin{bmatrix}
-3.5513 & 0.1664 \\ 0.1512 & -3.4908
\end{bmatrix} , \notag \\
S & = \begin{bmatrix}
0.0036 & -0.0144 \\ -0.0024 & 0.0152
\end{bmatrix},  \quad M   = - 4, \quad  B_{w} = 1 . \notag  
\end{align}

Then, we apply the following optimization problem to obtain an estimate of the system's ultimate-bound set
\begin{equation} \label{eqn:ex-opt}
\max_{P_1,\ldots,P_4} \ \min_i \ g_i :
\left\{ 
\begin{aligned}
& P_i - g_i I \geq 0, \ i=1,\ldots,4, \\[2mm]
& \qquad g_i > 0 \mbox{ \ and \eqref{eqn:thm1-LMI-2}.}
\end{aligned}
\right. 
\end{equation}

For $\beta = 1.98$ and $\alpha = 0.01$, we obtain
\begin{align}
P_1 & \!=\! \diag\{0.3599,0.3912\} , \   
P_2 \!=\! \diag\{0.4267,0.4546\} , \notag \\
P_3 & =  \begin{bmatrix} 0.5741 & 0.0152 \\ 0.0152 & 0.5800 \end{bmatrix}, \quad  
P_4 = 0.4144 \notag 
\end{align}
and
\begin{multline} \notag
\mcr = \{  (\xi,\partial_z \xi,v,\eta) \in \mcl_2^n \times \mcl_2^n \times \bbr^{n_v} \times \bbr^{n_\eta} : \\
\|(\xi,v,\eta)\|_X \leq 2.7485 \} .  
\end{multline}
Notice that \eqref{eqn:ex-opt} minimizes the size of $\mcr$ by maximizing $\epsilon_1$ as defined in \eqref{eqn:proof-thm1-0}. More precisely, we maximize the minimum of the smallest eigenvalues of $P_1,\ldots,P_4$ in order to obtain an optimized estimate of $\mcu_{\mathcal B}$. Fig.~\ref{fig:resultado_exemplo1} shows the $\|(\xi,\;\partial_{z}\xi,\; v,\;\eta)\|_{X}$ norm of the system and, as expected, the system's trajectory belongs to $\mcr$ as $t \rightarrow \infty$.
\begin{figure}
    \centering
    \includegraphics[width=0.8\linewidth]{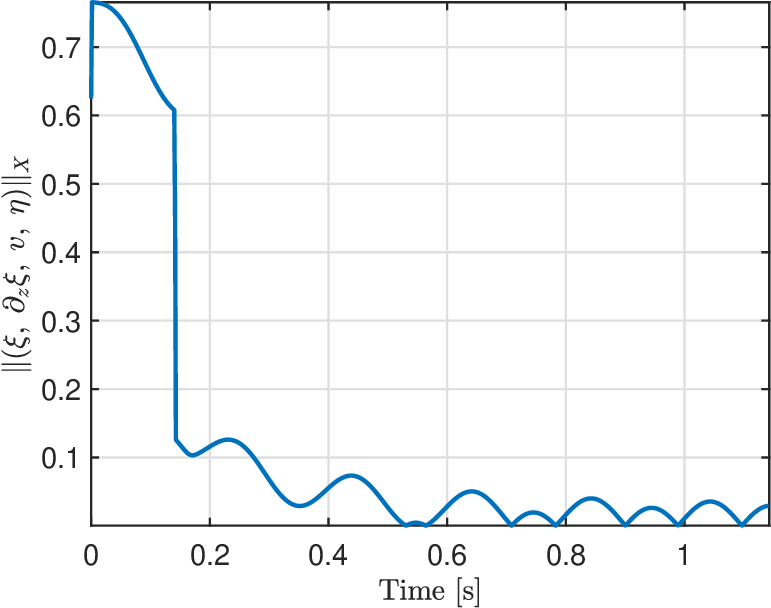}
    \caption{$\|(\xi,\;\partial_{z}\xi,\; v,\;\eta)\|_{X}$ of the system \eqref{eqn:pde-1}-\eqref{eqn:ode-2} with the parameters defined in \eqref{eq:parametrosl7_exemplo1}.}
    \label{fig:resultado_exemplo1}
\end{figure}

\subsection{Control Design}

Let the coupled PDE/ODE system as given in \eqref{eqn:pde-1}-\eqref{eqn:ode-2} be defined by  
\begin{align}
\Lambda  & = \begin{bmatrix}
               6  &  0  \\
               0  & 6
\end{bmatrix}, \quad  H = \begin{bmatrix}
               0     &  1  \\
               -1.1 & 0
\end{bmatrix},\label{eq:example2_paraml1} \\ 
B_{v} & = \frac{1}{10} \begin{bmatrix} 
    1 & 0 \\ 0 & 1 \end{bmatrix} ,  \ 
D_v = \begin{bmatrix} 0.1 & 0.0 \\ 0.0 & 0.1 \end{bmatrix} ,\\
 B_u &= B_w \!=\! \begin{bmatrix} 1 & 0 \\ 0 & 1 \end{bmatrix}
 M  = - \begin{bmatrix} 12  & 0 \\ 0 & 12  \end{bmatrix} \label{eq:example2_paraml3}. 
\end{align}

The above system is open-loop unstable, and we design a stabilizing controller that minimizes the resulting closed-loop $\mcu_{\mathcal B}$ by applying Theorem~\ref{thm:h1-iss-bcd} leading to
\begin{align}
 K_1 & \cong 0_{2} , \quad K_2 = \begin{bmatrix}
-0.0061 & -0.9959 \\ 1.0957 & -0.0064     
 \end{bmatrix} , 
\label{eq:example2_controll1} 
\end{align} 
\begin{align}
R & = \begin{bmatrix}
-1.9728 &  0.0000 \\ 0.0000 & -1.9728     
 \end{bmatrix} , \quad S \cong 0_n ,   \notag
\end{align}
where we have considered the following optimization problem
\begin{equation} \label{eqn:ex-opt-2}
\min_{Q_1,Q_2,P_3,P_p} \max_i g_i :
\left\{ \!
\begin{aligned}
& g_i I_n - Q_i \geq 0, ~i=1,2, \\[2mm]
& \begin{bmatrix} P_3 & I_n \\ I_n & - g_3 I_n \end{bmatrix} \geq 0, \\[2mm]
& \begin{bmatrix}
    P_p & G \\ G & g_4 I_n
\end{bmatrix} \geq 0 \\[2mm]
& \ g_j > 0, ~j=1,\ldots,4, \\[2mm]
& \mbox{ \ and \eqref{eqn:thm1-LMI-2-bcd}.}
\end{aligned}
\right. 
\end{equation}
with $\beta = 1.98$ and $\alpha = 0.11$. The resulting $\mcu_{\mathcal B}$ estimate is given by
\begin{multline} \notag
\mcr = \{  (\xi,\partial_z \xi,v,\eta) \in \mcl_2^n \times \mcl_2^n \times \bbr^{n_v} \times \bbr^{n_\eta} : \\ \|(\xi,v,\eta)\|_X \leq 1.5281 \} .  
\end{multline}
The above optimization problem maximizes $\epsilon_1$ as defined in \eqref{eqn:proof-thm1-0} by minimizing the largest eigenvalue among $P_1^{-1},\ldots,P_4^{-1}$ in order to obtain an optimized estimate of the system's ultimate-bound set. For illustrative purposes, Fig.~\ref{fig:resultado_exemplo2} depicts the H$_1$-norm of the system in closed loop. Notice that $\|(\xi,\;\partial_{z}\xi,\; v,\;\eta)\|_{X}\leq 0.04$, which shows that the steady state trajectory of the closed-loop system belongs to $\mcr$.
\begin{figure}
    \centering
    \includegraphics[width=0.8\linewidth]{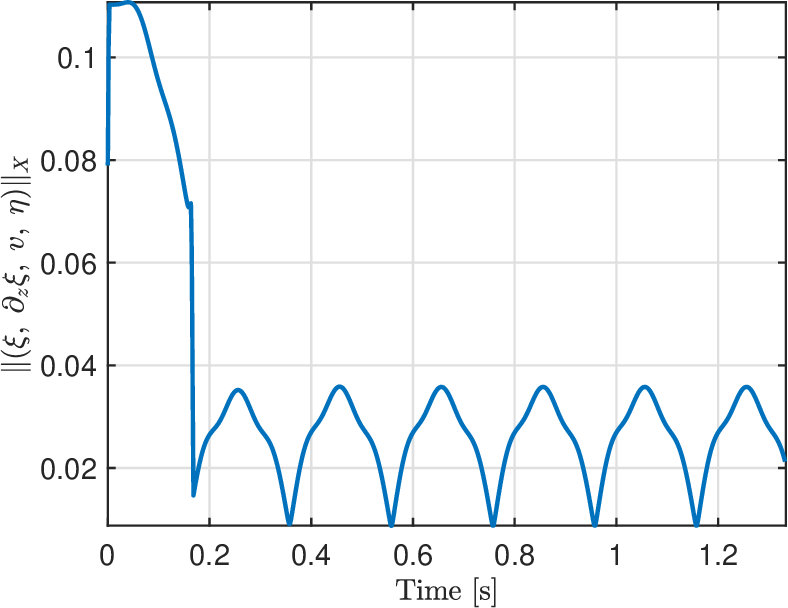}
    \caption{$\|(\xi,\;\partial_{z}\xi,\; v,\;\eta)\|_{X}$ of the system \eqref{eqn:pde-1}-\eqref{eqn:ode-2} with the parameters defined in \eqref{eq:example2_controll1}.}
    \label{fig:resultado_exemplo2}
\end{figure}

\section{Concluding Remarks} \label{sec:cr}

This paper investigated input-to-state stability and stabilization problems in the H$_1$-sense for linear systems governed by coupled ordinary differential equations and hyperbolic partial differential equations, subject to boundary actuation, sensing and disturbance inputs. First, an existing result from the literature is extended to ensure ISS under magnitude-bounded disturbances within the H$_1$-framework, relaxing restrictive assumptions on disturbance inputs. Then, stability analysis and control design conditions are formulated as linear matrix inequality constraints which ensure the ISS of the closed-loop system while providing an estimate of the system's ultimate-bound set. Numerical examples illustrated the potentials of the proposed methodology. Future research will be concentrated on dealing with non-homogeneous Semilinear Hyperbolic Systems. 

\bibliography{sample}

\end{document}